\documentclass[10pt, reqno]{amsart} 
\usepackage[english]{babel}
\usepackage{graphicx,epsfig}
\usepackage{bbm}
\usepackage[noadjust]{cite}
\usepackage{amsmath,amssymb,latexsym, amsfonts, amscd, amsthm}
\usepackage{enumerate}
\usepackage{mathrsfs}
\usepackage[T1]{fontenc}
\usepackage{lmodern}

\usepackage{mathtools}
\usepackage[margin=1.2in]{geometry}

\usepackage[dvipsnames]{xcolor}

\usepackage{xspace} %%%space after newcommand

\usepackage{cellspace}
\usepackage{tikz-cd}
\usetikzlibrary{cd}

\usepackage[mathscr]{euscript} 

\usepackage{caption}

\usepackage{relsize}
\usepackage[bbgreekl]{mathbbol}
\DeclareSymbolFontAlphabet{\mathbb}{AMSb} %to ensure that the meaning of \mathbb does not change
\DeclareSymbolFontAlphabet{\mathbbl}{bbold}
\newcommand{\Prism}{{\mathlarger{\mathbbl{\Delta}}}}

\makeindex 

\definecolor{ggreen}{rgb}{0.07, 0.55, 0.0}
\definecolor{rred}{rgb}{0.86, 0.18, 0.10}
\definecolor{bblue}{RGB}{15,70,150}

\usepackage[linktocpage=true]{hyperref}
\hypersetup{pdftoolbar=true, pdftitle={WachRamification}, pdfauthor={Pavel Coupek}, pdffitwindow=true, colorlinks=true, citecolor=ggreen, filecolor=ggreen, linkcolor=rred, urlcolor=ggreen, hypertexnames=false}

\usepackage{tablefootnote}

\newtheoremstyle{vetalike}% name
{8pt}% Space above
{6pt}% Space below
{\slshape}% Body font
{}% Indent amount
{\bfseries}% Theorem head font
{.}% Punctuation after theorem head
{1em}% Space after theorem head
{}% Theorem head spec (can be left empty, meaning `normal')

\newtheoremstyle{poznslike}% name
{8pt}% Space above
{6pt}% Space below
{}% Body font
{}% Indent amount
{\bfseries}% Theorem head font
{.}% Punctuation after theorem head
{\newline}% Space after theorem head
{}% Theorem head spec (can be left empty, meaning `normal')

\newtheoremstyle{otherlike}% name
{8pt}% Space above
{6pt}% Space below
{}% Body font
{}% Indent amount
{\bfseries}% Theorem head font
{.}% Punctuation after theorem head
{1em}% Space after theorem head
{}% Theorem head spec (can be left empty, meaning `normal')

\newtheoremstyle{deflike}% name
{8pt}% Space above
{6pt}% Space below
{}% Body font
{}% Indent amount
{\bfseries}% Theorem head font
{.}% Punctuation after theorem head
{1em}% Space after theorem head
{}% Theorem head spec (can be left empty, meaning `normal')

\theoremstyle{vetalike}
\newtheorem{thm}{Theorem}[section]
\newtheorem{prop}[thm]{Proposition}
\newtheorem{lem}[thm]{Lemma}
\newtheorem{cor}[thm]{Corollary}

\theoremstyle{poznslike}

\theoremstyle{deflike}
\newtheorem{deff}[thm]{Definition}
\newtheorem{nott}[thm]{Notation}

\newtheorem{rem}[thm]{Remark}

\theoremstyle{otherlike}

\newtheorem{pr}[thm]{Example}

\newcommand{\ainf}{A_{\mathrm{inf}}}

\newcommand{\BK}{\mathrm{BK}}
\newcommand{\et}{\text{\'{e}t}}

\newcommand{\Cr}[1]{\textnormal{($\mathrm{Cr}_{#1}$)}\xspace}

\newcommand{\oh}{\mathcal{O}}

\renewcommand{\H}{\mathrm{H}}
\newcommand{\R}{\mathsf{R}}
\renewcommand{\L}{\mathsf{L}}

\newcommand{\CC}{\mathbb{C}}
\newcommand{\QQ}{\mathbb{Q}}

\newcommand{\ZZ}{\mathbb{Z}}

\newcommand{\RR}{\mathbb{R}}
\renewcommand{\AA}{\mathbb{A}}
\newcommand{\NN}{\mathbb{N}}
\renewcommand{\aa}{\mathfrak{a}}

\newcommand{\Mat}{\mathrm{Mat}}

\newcommand{\Ker}{\mathrm{Ker}\,}

\newcommand{\gal}{\mathrm{Gal}}

\newcommand{\Es}{\mathfrak{S}}

\renewcommand{\theequation}{\arabic{section}.\arabic{equation}}


\begin{document}

\author{Pavel \v{C}oupek}
\address[Original Affiliaton]{Department of Mathematics, Michigan State University, Wells Hall, 619 Red Cedar Road, East Lansing, MI 48824, USA}
\address[Current Address]{Department of Mathematics, University of Virginia, Kerchof Hall, 141 Cabell Dr., Charlottesville, VA 22903, USA}
\email{kym4jc@virginia.edu}
\date{}
\title{Ramification bounds via Wach modules and $q$--crystalline cohomology}
\begin{abstract}
\noindent Let $K$ be an absolutely unramified $p$--adic field. We establish a ramification bound, depending only on the given prime $p$ and an integer $i$, for mod $p$ Galois representations associated with Wach modules of height at most $i$. Using an instance of $q$--crystalline cohomology (in its prismatic form), we thus obtain improved bounds on the ramification of $\H^{i}_{\et}(\mathscr{X}_{\mathbb{C}_K}, \mathbb{Z}/p\mathbb{Z})$ for a smooth proper $p$--adic formal scheme  $\mathscr{X}$ over $\oh_K$, for arbitrarily large degree $i$.
\end{abstract}
\maketitle

\tableofcontents

\section{Introduction}

Let $p>0$ be a prime, let $K$ be a $p$--adic field and denote by $G_K$ the absolute Galois group of $K$. The aim of this note is to study properties of mod $p$ $G_K$--representations $T$ that are crystalline in a suitable sense. While the optimal definiton of ``crystalline'' in this context is open to some discussion (see e.g \cite[p. 509]{BhattScholzeCrystals}), the intended meaning  for our purposes is one of the following two variants (relative to a fixed choice of an integer $i\geq 0$):
\begin{enumerate}[(a)]
\item{(abstract) $T$ is a $p$--torsion subquotient of a $G_K$--stable lattice in a crystalline $\QQ_p$--representation whose Hodge--Tate weights are contained in the interval $[-i, 0]$.}\label{TorsionCrysAbstract}
\item{(geometric) $T$ is the $i$--th \'{e}tale cohomology with $\ZZ/p\ZZ$--coefficients of the geometric generic fiber of a proper smooth $p$--adic formal scheme over $\oh_K$ (or a subquotient thereof).}\label{TorsionCrysGeometric}
\end{enumerate}

The two notions are related, but distinct. Notably, the geometric case is \textit{not} subsumed by the abstract case in any obvious way: the natural idea to express $\H^i_{\et}(X_{\CC_K}, \ZZ/p\ZZ)$ as the quotient of $\H^i_{\et}(X_{\CC_K}, \ZZ_p)$ is obstructed by the presence of torsion in $\H^i_{\et}(X_{\CC_K}, \ZZ_p)$ and $\H^{i+1}_{\et}(X_{\CC_K}, \ZZ_p)$ \cite{EmertonGee1}.

More concretely, we are interested in ramification of such representations. Let $G_K^v$ denote the upper--index higher ramification subgroups of $G_K=\gal(\overline{K}/K)$. Our main result is the following:

\begin{thm}[Theorem~\ref{thm:conclusion}]\label{thm:MainThm}
Assume that $K$ is absolutely unramified. Let $T$ be a mod $p$ crystalline representation in the sense of (\ref{TorsionCrysAbstract}) or (\ref{TorsionCrysGeometric}) above, relative to the integer $i$. Then $G_K^{v}$ acts trivially on $T$ when
$$v>\alpha+\max\left\{\,0,\; \frac{ip}{p^\alpha(p-1)}-\frac{1}{p-1}\,\right\},$$
where $\alpha$ is the least integer satisfying $p^{\alpha}>ip/(p-1)$.   
\end{thm}

Results of this type have a long history, going back to Fontaine's paper \cite{Fontaine} on the non--existence of Abelian varietes over $\QQ$ with good reduction everywhere. To a large extent, Fontaine's proof is based on a similar type of ramification bound for finite flat group schemes of order $p^n$ (over $\oh_K$ for general $K$). Subsequently, Fontaine \cite{Fontaine2} and Abrashkin \cite{Abrashkin} provided another version of ramification bounds for crystalline mod $p$ (mod $p^n$ in \cite{Abrashkin}) representations in the sense of (\ref{TorsionCrysAbstract}) above, but only when $K$ is absolutely unramified and the bounding integer $i$ satisfies $i<p-1$. The reason for these restrictions is the use of Fontaine--Laffaille theory \cite{FontaineLaffaille}, which works well only in this setting.

Of the further developments \cite{BreuilMessing, Hattori, Abrashkin2, CarusoLiu, Caruso}, let us explicitly list the extensions to the ``abstract semistable'' case, that is, the analogue of (\ref{TorsionCrysAbstract}) for semistable representations. Breuil \cite{BreuilLetter} (see also \cite{BreuilMessing}) proved such bounds assuming $ie<p-1$ where $e$ is the ramification index of $K/\QQ_p$, and under additional assumptions (Griffiths transversality). Hattori's work \cite{Hattori} then removed these extra assumptions and improved the applicable range to $i<p-1$ (with $e$ arbitrary, also in mod $p^n$ version). Finally, Caruso and Liu \cite{CarusoLiu} obtained a bound for abstract $p^n$--torsion semistable representations with $e$ and $i$ arbitrary, using the theory of $(\varphi, \widehat{G})$--modules \cite{LiuPhiG}, an enhancement of Breuil--Kisin modules \cite{Kisin} attached to lattices in semistable representations. 

It is worth noting that the above results also apply to the geometric  setting (\ref{TorsionCrysGeometric}), resp. its semistable analogue, using various comparison theorems \cite{FontaineMessing, CarusoLogCris, LiLiu}; however, these typically apply only when $ie<p-1$. This was the motivation for the author's previous work \cite{Ram1}, where a ramification bound was established for mod $p$ geometric crystalline representations with $e$ and $i$ arbitrary. 

While this has been achieved, the obtained result is not optimal: namely, in the setting $ie<p-1$ where the bounds of \cite{Hattori, CarusoLiu} apply to \'{e}tale cohomology of varieties with semistable reduction, the bound of \cite{Ram1} essentially agrees with these semistable bounds. A related question is raised in \cite{CarusoLiu} where the authors wonder whether there exists a general ramification bound for subquotients of a crystalline representation. They point out that they do not know any such genuinely crystalline bound beyond the results \cite{Fontaine2, Abrashkin} in the Fontaine--Laffaille case.

It is precisely these questions that motivate the present work: while we restrict to the absolutely unramified case ($e=1$), Theorem~\ref{thm:MainThm} applies for arbitrarily large $i$, in both the abstract and the geometric setting, hence goes beyond the scope of Fontaine--Laffaille theory. Moreover, specializing the results of \cite{Hattori, CarusoLiu, Ram1} to $e=1$, the present bound is in fact stronger (see Remark~\ref{rem:CrystallineVsSemistable}).  

Just like in \cite{Ram1}, the key input for the (geometric part of) the proof comes from prismatic cohomology \cite{BMS1, BMS2, BhattScholze}. Let us contrast the two approaches. In \cite{Ram1}, to the geometric mod $p$ crystalline representation a pair $(M_{\BK}, M_{\inf})$ was attached, consisting of the mod $p$ versions of Breuil--Kisin and $\ainf$--cohomology (the latter carrying Galois action). The key step in implementing a variant of the strategy \cite{CarusoLiu} was then to prove a series of conditions $\Cr{s}$, $s \geq 0$, reflecting the crystalline origin of these modules. These conditions are of the form ``$(g-1)M_{\BK} \subseteq I_s M_{\inf}$'' for $g$ coming from Galois groups of members of the Kummer tower $\{K(\pi^{1/p^s})\}_s$ associated with the Breuil--Kisin prism $\Es$ and its embedding to the Fontaine prism $\ainf$. 

In contrast, the present paper uses the theory of Wach modules \cite{Wach, Wach2, ColmezWachM, BergerWachM} rather than Breuil--Kisin modules. On the cohomological side, Breuil--Kisin cohomology is replaced by an instance of $q$--crystalline cohomology \cite[\S 16]{BhattScholze}, which we call \textit{Wach cohomology}. Such a replacement is natural: unlike Breul--Kisin modules, Wach modules relate only to crystalline representations. 

Roughly speaking, the shift from Breuil--Kisin to Wach modules amounts to replacing the prism $\Es$ by the Wach prism $\AA \subseteq \ainf$, which has many consequences. Firstly, the Kummer tower $\{K(\pi^{1/p^s})\}_s$ is replaced by the better--behaved cyclotomic tower $\{K(\mu_{p^s})\}_s$ in our argument. This is what in the end allows us to obtain a stronger ramification bound. Secondly, unlike $\Es$, the subring $\AA$ of $\ainf$ is stable under Galois action, and ultimately, so are Wach modules. This allows us to replace the use of conditions $\Cr{s}$ by the single condition analogous to $\Cr{0}$, which is in fact part of the definition of a Wach module. On the other hand, the theory of Wach modules works well only for $K$ absolutely unramified, which is why we consider only this case.

The outline of the paper is as follows. In Section~\ref{sec:prelim} we introduce  the most relevant background and notation on prisms that we use, as well as notation connected with ramification groups and ramification bounds. The notion of Wach modules, or rather a version of it suited for our purposes, is recalled in Section~\ref{sec:WachCohomology}. Here we also define (mod $p$) Wach cohomology groups and relate them to \'{e}tale cohomology. Finally, in Section~\ref{sec:RamBound}, we carry out the proof of Theorem~\ref{thm:MainThm}. We end the paper by an example showing that our bound in general \textit{does not} apply to semistable representations.

\vspace{1em}
\noindent \textbf{Acknowledgement.} As will become apparent, the present note is greatly inspired by the work \cite{CarusoLiu} of Xavier Caruso and Tong Liu. I am in particular very grateful to Tong Liu for his input through various discussions on this topic, and overall for his encouragement in carrying out this work. I also thank the anonymous referee for pointing out an error in an earlier version of the article, and for his other helpful suggestions that greatly improved clarity and exposition of the article. During the preparation of the paper, the author was partially supported by the NSF grant DMS-2337830 (PI: Preston Wake).

\section{Preliminaries}\label{sec:prelim}
\subsection{Prisms $\AA$ and $\ainf$}
We fix a prime $p$ throughout. Let $k$ be a perfect field of characteristic $p$, and let $K=W(k)[1/p]$ be the associated absolutely unramified $p$--adic field. Let us denote by $\CC_K$ the completed algebraic closure of $K$, and by $\oh_{\CC_K}$ its ring of integers. We let $G_K$ denote the absolute Galois goup of $K$.

For a detailed discussion of prisms and prismatic cohomology, we refer the reader to \cite{BhattScholze}. Here we only recall that a $p$-torsion free, bounded, oriented prism $(A, I)$ is given by a $p$-torsion free ring $A$, a principal ideal $(d)=I \subseteq A$ where $d \in A$ is a non-zero divisor, and a Frobenius lift $\varphi: A \rightarrow A$ (i.e., an endomorphism reducing to the absolute Frobenius modulo $p$) such that:
\begin{enumerate}
\item $A$ is $(p, d)$-adically complete.
\item The element $d \in A$ is distinguished, i.e., $(\varphi(d)-d^p)/p$ is a unit in $A$. Equivalently, we have $p \in (d, \varphi(d)).$
\item The ring $A/d$ has bounded $p^{\infty}$-torsion, i.e., $A/d[p^{\infty}]=A/d[p^{N}]$ for some $N \in \NN$.
\end{enumerate}

The examples of such prisms relevant to the present article are the following:

\begin{enumerate}
\item The central prism of interest is the prism $(\AA, I)$ where $\AA=W(k)[[q-1]]$ (with $q$ a formal variable, unit in $\AA$), and $I$ is the principal ideal generated by $$[p]_q=\frac{q^p-1}{q-1}=1+q+q^2+\dots +q^{p-1}.$$ 
 The Frobenius lift $\varphi$ on $\AA$ is given by the Witt vector Frobenius on $W(k)$ and by $\varphi(q)=q^p$.

\noindent When $W(k)=\ZZ_p$, this is the $q$--crystalline prism from \cite[Example~1.3 (4)]{BhattScholze}. To stress the connection with the theory of Wach modules, and to avoid the conflation with the case over $\ZZ_p$, we refer to $(\AA, I)$ as the \textit{Wach prism associated with $W(k)$}.

\item Another key prism is the Fontaine prism $(\ainf, \mathrm{Ker}\,  \theta)$ (an instance of a perfect prism \cite[Example~1.3 (2)]{BhattScholze}). Here $\ainf=W(\oh_{\CC_K^{\flat}})$ where $\oh_{\CC_K^{\flat}}$ is the inverse limit perfection of $\oh_{\CC_K}/p$. The map $\theta: \ainf \rightarrow \oh_{\CC_K}$ is the Fontaine's map, determined by sending the Teichm\"{u}ller lift $[x]$ of any element $x=(x_0 \,\mathrm{mod}\, p, x_0^{1/p} \,\mathrm{mod}\, p, \dots ) \in \oh_{\CC_K^{\flat}}$ to $x_0$.
\end{enumerate}

Let us fix a compatible system $(\zeta_{p^s})_s$ of primitive $p^s$--th roots of unity, which determines the element $$\varepsilon=(1, \zeta_p, \zeta_{p^2}, \dots) \in \oh_{\CC_K}^{\flat}.$$ 
There is a map $\AA\rightarrow \ainf$ given by sending $q-1$ to  $[\varepsilon^{1/p}]-1,$ 
and it can be shown that $\Ker{\theta}$ is generated by the image $\xi$ of $[p]_q$ under this map. We thus obtain a map of prisms $\AA\rightarrow \ainf$; modulo $I$, this map becomes the inclusion 
$$W(k)[[q-1]]/([p]_q)\simeq W(k)[\zeta_p]\rightarrow \oh_{\CC_K}.$$

\begin{lem}
The map $\AA \rightarrow \ainf$ is faithfully flat.
\end{lem}

\begin{proof}
Both $\AA$ and $\ainf$ are $p$--adically complete with $\AA$ Noetherian. Thus, to prove flatness, by \cite[Lemma~0912]{stacks} it is enough to show that $\AA/p^n \rightarrow \ainf/p^n$ is flat for every $n$. This latter statement is a special case of \cite[Proposition~2.2.12]{EmertonGee2}.
For faithful flatness, it is enough to observe that the (unique) maximal ideal of $\ainf$ lies above the unique maximal ideal $\mathfrak{m}_{\AA}=(p, q-1)$.\end{proof}

For an integer $s \in \NN$, denote by $K_{s}$ the extension of $K$ given by $K(\mu_{p^{s+1}})$. Further set $K_{\infty}=K(\mu_{p^\infty})=\bigcup_s K_{s}$.  For every $s \in \NN \cup \{\infty\}$, denote by $G_s$ the absolute Galois group of $K_{s}$. Further denote by $\Gamma$ the topological group $\gal(K_{\infty}/K)$. Then $\Gamma$ naturally identifies with the quotient $G_K/G_\infty$. For each $s \in \NN$, denote by $\Gamma_s$ the image of $G_s$ in $\Gamma\simeq G_K/G_\infty$. Then we similarly have $\Gamma_s=\gal(K_{\infty}/K_{s}).$ If $\chi: G_K \to \ZZ_p^{\times}$ denotes the cyclotomic character, then $\chi$ induces an isomorphism of $\Gamma$ with $\ZZ_p^{\times}$. Under this isomorphism, for every $s \in \NN$, $\Gamma_s$ corresponds to the subgroup $1+p^{s+1}\ZZ_p$ (of index $(p-1)p^{s}$ in $\ZZ_p^{\times}$).

\begin{rem}
In the case $p=2$, the behavior of the tower $K_s$ is slightly different from when $p$ is odd. Due to the fact that $K(\mu_2)=K$, the filtration of $G_K$ by the subgroups $G_{K(\mu_{2^t})}$ has first jump for $t=2$ rather than $t=1$ as in the odd case. In order to better keep track of differences that occur due to this, let us denote by $\delta_p$ the Kronecker delta $\delta_{2, p},$ i.e., $\delta_p=1$ when $p=2$ and  $\delta_p=0$ otherwise.
\end{rem}

With this convention established, let us fix a choice of an element $\gamma \in \Gamma$ such that $\chi(\gamma)=1+p^{1+\delta_p}$. Then $\gamma$ is a topoloical generator of $\Gamma_{\delta_p}$ (the first proper term of the filtration by $\{\Gamma_s\}_s$ of $\Gamma$), and for every $s \in \NN$, $\gamma^{p^s}$ topologically generates $\Gamma_{s+\delta_p}$. 

The group $\Gamma$ has natural action on $\AA$ via 
$$g(q)=q^{\chi(g)},\;\; g \in \Gamma,$$
making the map $\AA \rightarrow \ainf$ $G_K$--equivariant when treating the $\Gamma$--action as a $G_K$--action via the map $G_K \twoheadrightarrow G_K/G_{\infty}\simeq \Gamma$ (the $G_K$--action on $\ainf$ is induced by the one on $\oh_{\CC_K}$ by functoriality). In particular, we have the identity $\gamma(q)=q^{1+p^{1+\delta_p}}$.

\subsection{Ramification groups and Fontaine's property $(P_m)$}

For an algebraic extension $F/K$, denote by $v_F$ the additive valuation on $F$ normalized by $v_F(F^{\times})=\ZZ$. Given finite extensions $F/E/K$ with $F/E$ Galois, the lower--index numbering on ramification groups of $G=\gal(F/E)$ we consider is
$$G_{(\lambda)}=\{g \in G\;|\; v_F(g(x)-x)\geq \lambda\}, \;\; \lambda \in \RR_{\geq 0}.$$
Note that $G_{(\lambda)}=G_{\lambda-1},$ where $G_{\lambda}$ are the usual lower--index ramification groups as in \cite[\S IV]{SerreLocalFields}.

For $t \geq 0$, we define the Herbrand function $\varphi_{F/E}(t)$ by
$$\phi_{F/E}(t)=\int_0^{t}\frac{\mathrm{d}s}{[G_{(1)}:G_{(s)}]},$$
(which makes sense since $G_{(s)} \subseteq G_{(1)}$ for all $s>0$). This is an increasing, concave, piecewise linear function, and we define $\psi_{F/E}$ to be the inverse function of $\phi_{F/E}$. Then the upper--index ramification subgroups of $G$ are given by
$$G^{(u)}=G_{(\psi_{F/E}(u))}, \;\; u \in \RR_{\geq 0}.$$
Once again, this numbering is related to the numbering $G^{u}$ given in \cite[\S~IV]{SerreLocalFields} by $G^{(u)}=G^{u-1}$. In particular, the numbering $G^{(u)}$ is compatible with passing to quotients. Given a possibly infinite Galois extension $N/E$, we may therefore set
$$\gal(N/E)^{(u)}=\varprojlim_{M}\gal(M/E)^{(u)},$$
where $M$ ranges over finite Galois extensions $M/E$ contained in $N$.

Given an algebraic extension $M/K$ and a real number $m> 0$, we denote by $\aa_M^{> m}$ ($\aa_M^{\geq m}$, resp.) the ideal of all elements $x \in \oh_M$ with $v_K(x)>m$ ($v_K(x)\geq m,$ resp.). We consider the following condition formulated by Fontaine \cite{Fontaine}:

$$\begin{array}{cc}
(P_m^{F/E}): & \begin{array}{l}\text{For any algebraic extension }M/E,\text{ if there exists an }\oh_E\text{--algebra map}\\ \oh_F\rightarrow \oh_M/\mathfrak{a}_M^{>m},\text{ then there exists an }E\text{--algebra map }F\hookrightarrow M.\end{array}
\end{array}$$

Let us now assume that $F/E$ is finite. We let $\mu_{F/E}$ denote the infimum of all $u$ such that $\gal(F/E)^{(u)}=\{\mathrm{id}\}.$ We measure the ramification of $F/E$ in terms of the invariant $\mu_{F/E}$, which is closely connected with the property $(\mathrm{P}_m)$:

\begin{prop}[{\cite[Proposition 1.5]{Fontaine}, \cite[Propositions 2.2, 3.3]{Yoshida}, \cite[Proposition 4.2.1]{CarusoLiu}}]\label{prop:PmAndRamification}
Denote by $e_{E/K}$ the ramification index of $E/K$, and let $m>0$ be a real number. If $(\mathrm{P}_m^{F/E})$ holds then $\mu_{F/E}\leq  e_{E/K}m$. Moreover, the validity of $(\mathrm{P}_m^{F/E})$ and the value of $\mu_{F/E}$ remain unchanged if $E$ is replaced by any subfield $E'$ of $F$ unramified over $E$.
\end{prop} 

Finally, let us record a lemma on the behavior of $\mu$ in towers that will be useful later on.

\begin{lem}[{\cite[Lemma~4.3.1]{CarusoLiu}}]\label{lem:MuInTowers}
Let $L/F/E$ be a tower of finite Galois extensions. Then
$$\mu_{L/E}=\mathrm{max}\left\{\mu_{F/E}, \phi_{F/E}(\mu_{L/F})\right\}.$$
\end{lem}

\section{Wach cohomology and Wach modules} \label{sec:WachCohomology}

Recall that given a (bounded) prism $(A, I)$, to a smooth $p$--adic formal scheme $\mathscr{X}$ over $A/I$ one can associate the prismatic cohomology $\R\Gamma_{\Prism}(\mathscr{X}/A)$. Its cohomology groups $\H_{\Prism}^i(\mathscr{X}/A)$ are $A$--modules equipped with, among other structures, a $\varphi_A$--semilinear operator $\varphi$. When $\mathscr{X}$ is proper, $\R\Gamma_{\Prism}(\mathscr{X}/A)$ is represented by a perfect complex (see \cite[Theorem~1.8]{BhattScholze}).

In our setting, the relevant variants of prismatic cohomology are the following.

\begin{deff}
Consider a smooth proper $p$--adic formal scheme $\mathscr{X}$ over $W(k)$, and denote by $\mathscr{X}_p$ the base change of $\mathscr{X}$ to $W(k)[\zeta_p]$. By \textit{Wach cohomology} of $\mathscr{X}$, we mean the prismatic cohmology $\R\Gamma_{\Prism}(\mathscr{X}_p/\AA)$. The individual Wach cohomology groups are denoted by $\H^i_{\Prism}(\mathscr{X}_p/\AA)$. 

The \textit{mod $p$ Wach cohomology of $\mathscr{X}$} is given by $$\overline{\R\Gamma_{\Prism}(\mathscr{X}_p/\AA)}=\R\Gamma_{\Prism}(\mathscr{X}_p/\AA)\stackrel{\L}{\otimes}_{\ZZ}\ZZ/p\ZZ.$$
We denote by $\overline{\H^i_{\Prism}(\mathscr{X}_p/\AA)}$ the individual cohomology groups of the mod $p$ Wach cohomology of $\mathscr{X}$.
\end{deff}

Let us also recall a version of Wach modules suitable for our purposes. From the standard definitions \cite{Wach, ColmezWachM, BergerWachM}, it deviates in that we allow Wach modules that are not necessarily free.

\begin{deff}\label{def:WachModule}
A \textit{Wach module of height $\leq i$} is a finitely generated $\AA$--module $M$ endowed with a continuous, $\AA$--semilinear action of $\Gamma$ and a $\varphi_{\AA}$--semilinear, $\Gamma$-equivariant map $\varphi: M \rightarrow M$  satisfying the following:
\begin{enumerate}[(1)]
\item{The linearization $\varphi_{\mathrm{lin}}=1\otimes {\varphi}:\varphi_{\AA}^*M \rightarrow M$ of $\varphi$ admits a map $\psi:M \rightarrow \varphi_{\AA}^*M$ such that both $\varphi_{\mathrm{lin}} \circ \psi$ and $\psi \circ \varphi_{\mathrm{lin}}$ are given by multiplication by $([p]_q)^i$.}\label{def:WachModulePhi}
\item{The induced $\Gamma$--action on $M/(q-1)M$ is trivial. Equivalently, for every $g \in \Gamma$ we have
$$(g-1)M \subseteq (q-1)M.$$}\label{def:WachModuleCr}
\end{enumerate}
\end{deff}

We are in particular interested in the case when $M$ is annihilated by $p$, i.e. when $M$ is a module over $\AA/p \simeq k[[q-1]]$. We refer to $M$ as \textit{$p$--torsion Wach module} in this case.

Wach cohomology groups naturally give rise to Wach modules. As stated earlier, $\varphi$ comes directly from its description as prismatic cohomology. Let us now discuss the $\Gamma$--action portion of the data.

\begin{lem} \label{lem:GKInvarianceOfI}
The ideal $I=([p]_q) \subseteq \AA$ is stable under the $G_K$--action. 
\end{lem}

\begin{proof}
The map $\AA \rightarrow \AA/I \simeq  W(k)[\zeta_p]$ sends $q$ to $\zeta_p$; it is then easy to see that this map is $G_K$--equivariant. Therefore, the kernel $I$ is necessarily $G_K$--stable.
\end{proof}

For a smooth proper $p$--adic formal scheme $\mathscr{X}$ over $W(k)$, there is a natural $G_K$--action on $\R\Gamma_{\Prism}(\mathscr{X}_p/\AA)$ given as follows. For $g \in G_K,$ acting by $g$ gives a map of prisms $g:(\AA,[p]_q) \rightarrow (\AA, [p]_q)$. By base change of prismatic cohomology \cite[Thm 1.8 (5)]{BhattScholze}, we obtain an $\AA$--linear isomorphism
\begin{align}\label{ActionMap}
g^*\R\Gamma_{\Prism}(\mathscr{X}_p/\AA)\rightarrow \R\Gamma_{\Prism}(g^*\mathscr{X}_p/\AA)=\R\Gamma_{\Prism}(\mathscr{X}_p/\AA)
\end{align}
where the last identity comes from identifying $g^*\mathscr{X}_p$ with $\mathscr{X}_p$ via the canonical isomorphism $g^*\mathscr{X}_p\rightarrow \mathscr{X}_p$. The above map can then be identified with an $\AA$--$g$--semilinear map 
$$g: \R\Gamma_{\Prism}(\mathscr{X}_p/\AA)\rightarrow \R\Gamma_{\Prism}(g^*\mathscr{X}_p/\AA),$$ 
which is the action of $g$ on $\R\Gamma_{\Prism}(\mathscr{X}_p/\AA)$. When $g$ is from $G_{{\infty}},$ it acts trivially on both $\mathscr{X}_p$ and $\AA$; thus, the above action becomes an action of $G_K/G_{{\infty}},$ i.e., action of $\Gamma$.

The $G_K$--action on $\R\Gamma_{\Prism}(\mathscr{X}_{\oh_{\CC_K}}/\ainf)$ can be described similarly; consequently, it is easy to see that the base-change map 
$$\R\Gamma_{\Prism}(\mathscr{X}_p/\AA)\stackrel{\L}{\widehat{\otimes}_{\AA}}\ainf\rightarrow \R\Gamma_{\Prism}(\mathscr{X}_{\oh_{\CC_K}}/\ainf)$$
of \cite[Theorem~1.8 (5)]{BhattScholze} is $G_K$--equivariant (of course, this action no longer factors through $\Gamma$).

There is a complex $C^{\bullet}(\mathscr{X}_p)$ ($C^{\bullet}(\mathscr{X}_{\oh_{\CC_K}}),$ resp.) modelling $\R\Gamma_{\Prism}(\mathscr{X}_p/\AA)$ ($\R\Gamma_{\Prism}(\mathscr{X}_{\oh_{\CC_K}}/\ainf)$, resp.) with the following properties:
\begin{enumerate}[(1)]
\item{\label{CechFlat}($C^{\bullet}(\mathscr{X}_p)$ ($C^{\bullet}(\mathscr{X}_{\oh_{\CC_K}}),$ resp.) is a perfect complex and consists termwise of flat $\AA$-modules ($(p, \xi)$--completely flat $\ainf$--modules, resp.),}
\item{\label{CechGaction} The $G_K$--action on $\R\Gamma_{\Prism}(\mathscr{X}_p/\AA)$ ($\R\Gamma_{\Prism}(\mathscr{X}_{\oh_{\CC_K}}/\ainf)$, resp.) comes from a (``strict'') $G_K$--action on $C^{\bullet}(\mathscr{X}_p)$ ($C^{\bullet}(\mathscr{X}_{\oh_{\CC_K}}),$ resp.). In more detail, for every $g \in G_K$ there is an isomorphism $g^*C^{\bullet}(\mathscr{X}_p)\stackrel{\sim}{\rightarrow}C^{\bullet}(g^*\mathscr{X}_p)=C^{\bullet}(\mathscr{X}_p)$ that represents the map (\ref{ActionMap}), and which defines a semilinear action of $G_K$ on $C^{\bullet}(\mathscr{X}_p)$ (similarly for $\R\Gamma_{\Prism}(\mathscr{X}_{\oh_{\CC_K}}/\ainf)$). Moreover, $G_{\infty}$ acts trivially on $C^{\bullet}(\mathscr{X}_p)$, hence we get a $\Gamma$--action on $C^{\bullet}(\mathscr{X}_p)$.}
\item{We have $C^{\bullet}(\mathscr{X}_{\oh_{\CC_K}}) \simeq C^{\bullet}(\mathscr{X}_p){\widehat{\otimes}_{\AA}}\ainf$, compatibly with the $G_K$--action. Here the completed tensor product is computed term--by--term.}
\end{enumerate}

The existence of such complexes was established in \cite[\S 2.2]{Ram1} by extending a Čech--Alexander construction of \cite{BhattScholze} from affine case to the case of a general separated formal scheme. In the case of the complex $C^{\bullet}(\mathscr{X}_{\oh_{\CC_K}}),$ the presence of the Galois action is then explained in detail in \cite[\S 4.1]{Ram1}. In the case of the complex $C^{\bullet}(\mathscr{X}_p)$, the Galois action is established completely analogously.

The following proposition is the analogue of the condition $\Cr{0}$ from \cite{Ram1} in our present context.

\begin{prop}\label{prop:crys}
For all $i$ and all $g \in \Gamma,$ we have $$(g-1)C^i(\mathscr{X}_p)\subseteq (q-1)C^i(\mathscr{X}_p).$$
\end{prop}

\begin{proof}
Consider the complex $C^\bullet(\mathscr{X}_p)/(q-1)C^\bullet(\mathscr{X}_p)$, where the quotient is computed term-by-term. This is the \v{C}ech--Alexander complex computing $\R\Gamma_{\Prism}(\mathscr{X}_k/W(k)),$ that is, up to $\varphi_{W(k)}$--twist, the crystalline cohomology of the special fiber. The $\Gamma$--action on this complex, defined as in (\ref{CechGaction}), on one hand comes from $C^\bullet(\mathscr{X}_p)$, and on the other is trivial as $G_K$ acts trivially on both $W(k)$ and $\mathscr{X}_k$. This proves the claim.
\end{proof}

As a consequence of (\ref{CechFlat}) above, $\overline{\R\Gamma_{\Prism}(\mathscr{X}_p/\AA)}$ is modelled by the complex $C^{\bullet}(\mathscr{X})/pC^{\bullet}(\mathscr{X})$ (computed term--by--term). Then we have

\begin{cor}\label{cor:WachOnCohomology}
If $M$ is either $\H^i_{\Prism}(\mathscr{X}_p/\AA)$ or $\overline{\H^i_{\Prism}(\mathscr{X}_p/\AA)}, $ we have for all $g \in \Gamma$ $$(g-1)M \subseteq (q-1)M.$$
Consequently, $M$ is a Wach module of height $\leq i$.
\end{cor}

\begin{proof}
Let $C^{\bullet}$ be either the complex $C^{\bullet}(\mathscr{X})$ or $C^{\bullet}(\mathscr{X})/pC^{\bullet}(\mathscr{X}).$ It suffices to prove the condition $(g-1)Z^i \subseteq (q-1)Z^i$ where $Z^i$ denotes the degree $i$ cocycles in $C^{\bullet}$. In both cases, $q-1$ is a non--zero divisor on $C^j$ for every $j$, since $C^{j}(\mathscr{X})$ is $\AA$--flat and $p, q-1$ is a regular sequence on $\AA$.

Given $c \in Z^i$, by Proposition~\ref{prop:crys} we have $(g-1)c=(q-1)c'$ for some $c' \in C^i$, and it is enough to observe that $c'\in Z^i.$ This is indeed the case: If $\partial$ denotes the differential $C^{i}\rightarrow C^{i+1},$ we have
$$(q-1)\partial(c')=\partial((q-1)c')=\partial((g-1)c)=(g-1)\partial(c)=0,$$
and we may conclude that $\partial(c')=0$ since $q-1$ is a non--zero divisor on $C^{i+1}$.

This verifies condition~(\ref{def:WachModuleCr}) of Definition~\ref{def:WachModule}, while condition~(\ref{def:WachModulePhi}) is a general fact about prismatic cohomology \cite[Theorem~1.8 (6)]{BhattScholze}. It follows that $M$ is a Wach module of height $\leq i$.
\end{proof}

We also need some better control on the action when acting by elements of the subgroup $\Gamma_s\subseteq \Gamma$. In \cite{Ram1}, this was done using certain somewhat independent conditions $\Cr{s}$. In the context of Wach modules, we can obtain the control as a consequence of the  property (\ref{def:WachModuleCr}) of Definition~\ref{def:WachModule}.

\begin{lem}\label{lem:IterateCrys}
If $M$ is an $\AA/p$--module with a semilinear $\Gamma_{\delta_p}$--action satisfying $(\gamma-1)M \subseteq (q-1)M$, then the same is true of the module $M'=(q-1)M$.
\end{lem}

\begin{proof}
For a given $m \in M$, note that 
\begin{align*}(\gamma-1)((q-1)m)&=(\gamma(q)-1)\gamma(m)-(q-1)m\\
&=(\gamma(q)-1)\gamma(m)-(q-1)\gamma(m)+(q-1)\gamma(m)-(q-1)m \\
&=(\gamma(q)-q)\gamma(m)+(q-1)(\gamma(m)-m)\\
&=q(q-1)^{p^{1+\delta_p}}\gamma(m)+(q-1)^2m'
\end{align*}
for some $m' \in M$, where on the last line,  $pM=0$ is used in order to replace $(q^{p^{1+\delta_p}}-1)$ with $(q-1)^{p^{1+\delta_p}}$. Thus, we have $(\gamma-1)((q-1)M) \subseteq (q-1)^2M,$ as desired.
\end{proof}

\begin{prop}\label{prop:CrystallineConditionAnalogue}
For a $p$--torsion Wach module $M$, we have $$\forall g \in \Gamma_{s+\delta_p}: (g-1)M \subseteq (q-1)^{p^s}M\;. $$
\end{prop}

\begin{proof} 
When $s=0$, this follows from part (\ref{def:WachModuleCr}) of Definition~\ref{def:WachModule}, so we may assume $s\geq 1$. Since $\Gamma_{s+\delta_p}$ is topologically generated by $\gamma^{p^s},$ it is enough to show the assertion for $\gamma^{p^s}$. Observe that $\gamma^{p^s}-1=(\gamma-1)^{p^s}$ as endomorphisms of $M$ since $pM=0$. Thus, we need to verify
$$(\gamma-1)^{p^s}M\subseteq (q-1)^{p^s}M.$$
But this follows by repeated use of Lemma~\ref{lem:IterateCrys}.
\end{proof}

Finally, let us discuss Galois representations attached to Wach modules in the sense of Definition~\ref{def:WachModule}.

\begin{deff}
The $G_K$--module associated with a $p$--torsion Wach module $M$ is given by $$T(M)=(M \otimes_{\AA}\CC_K^{\flat})^{\varphi=1},$$
where the map $\AA \rightarrow \oh_{\CC_K^{\flat}}\rightarrow \CC_K^{\flat}$ is given by sending $q$ to $\varepsilon^{1/p}$.
\end{deff}

In the geometric setting, the representation obtained this way is the appropriate \'{e}tale cohomology.

\begin{prop}\label{prop:EtaleRealizationGeometric}
\begin{enumerate}[(1)]
\item{Let T be a $\ZZ/p\ZZ[G_K]$-module of the form $L/pL$ for a $G_K$--stable lattice $L$ in a crystalline representation whose Hodge--Tate weights are in the range $[-i, 0]$. Then $T=T(M)$ for a $p$--torsion Wach module $M$.}\label{WachModForLatticeQuotient}
\item{For a proper smooth $p$--adic formal scheme $\mathscr{X}$ over $W(k)$, we have the following:
\begin{enumerate}\item{$T(\overline{\H^i_{\Prism}(\mathscr{X}_p/\AA)})=\H^i_{\et}(X_{\CC_K}, \ZZ/p\ZZ)\, ,$}\label{FpCohomology} 
\item{$T(\H^i_{\Prism}(\mathscr{X}_p/\AA)/p\H^i_{\Prism}(\mathscr{X}_p/\AA))=\H^i_{\et}(X_{\CC_K}, \ZZ_p)/p\H^i_{\et}(X_{\CC_K}, \ZZ_p)\,.$}\label{EtaleModpCohomology}
\end{enumerate} }\label{WachModForCohomology}
\end{enumerate} 

\end{prop}

\begin{proof}
Let us start with (\ref{WachModForCohomology}). The claim (\ref{FpCohomology}) is proved the same way as \cite[Lemma 2.1.6]{KisinEssDim} in the setting of Breuil--Kisin cohomology. First, note that $M \mapsto T(M)$ factors as the base chage $M \rightarrow M\otimes_{\AA}\ainf$ to $\ainf$--cohomology followed by the analogous functor $M_{\inf}\mapsto (M_{\inf}\otimes_{\ainf}\CC_K^{\flat})^{\varphi=1}$ of mod $p$ Breuil--Kisin--Fargues modules. Using \cite[Theorem~1.8]{BhattScholze}, we obtain a long exact sequence 
\begin{center}
\begin{tikzcd}
\cdots \H^i_{\et}(X_{\CC_K}, \ZZ/p\ZZ) \ar[r] &  \overline{\H^i_{\Prism}(\mathscr{X}_p/\AA)}\otimes_{\AA}\CC^{\flat}_K \ar[r, "1-\varphi"] & \overline{\H^i_{\Prism}(\mathscr{X}_p/\AA)}\otimes_{\AA}\CC^{\flat}_K \ar[r] & \H^{i+1}_{\et}(X_{\CC_K}, \ZZ/p\ZZ)  \cdots, 
\end{tikzcd}
\end{center}
so it is enough to observe that the map $1-\varphi$ is surjective (for every $i$). Since $[p]_q$ becomes invertible in $\CC_K^{\flat}$, $ \overline{\H^i_{\Prism}(\mathscr{X}_p/\AA)}\otimes_{\AA}\CC^{\flat}_K $ is in fact an \'{e}tale $\varphi$--module, hence of the form $T\otimes \CC_K^{\flat}$ for some finite $\ZZ/p\ZZ$--module $T$, with $\varphi$ given by the Frobenius on $\CC_K^{\flat}$. It follows that $1-\varphi$ is surjective. 

Repeating a d\'{e}vissage version of the argument (or simply invoking \cite[Theorem 1.1 (vii)]{MorrowNotes} and \cite[\S 17]{BhattScholze})) shows that $T(\H^i_{\Prism}(\mathscr{X}_p/\AA))=\H^i_{\et}(X_{\CC_K}, \ZZ_p),$ where we extend the definition of $T$ to all Wach modules by the formula $T(M)=(M\otimes_{\AA}W(\CC_K^{\flat}))^{\varphi=1}$. The second claim now follows from the fact that $T$ takes the right--exact sequence of Wach modules
\begin{center}
\begin{tikzcd}
\H^i_{\Prism}(\mathscr{X}_p/\AA) \ar[r, "p"] & \H^i_{\Prism}(\mathscr{X}_p/\AA) \ar[r] & \H^i_{\Prism}(\mathscr{X}_p/\AA)/p\H^i_{\Prism}(\mathscr{X}_p/\AA)  \ar[r] & 0 
\end{tikzcd}
\end{center}
to a right--exact sequence again. This proves (\ref{EtaleModpCohomology}).

To prove (\ref{WachModForLatticeQuotient}), one uses the result of Berger \cite{BergerWachM} that the crystalline lattice $L$ is of the form $T(M_0)$ for a  Wach module $M_0$ of height $\leq i$, finite free as an $\AA$--module. Proceeding the same way as in the proof of (\ref{EtaleModpCohomology}), it follows that $T=T(M)$ for the mod $p$ Wach module $M=M_0/pM_0$.  
\end{proof}

\section{Ramification bound}\label{sec:RamBound}

We now proceed with the proof of the ramification bound. We follow the strategy used in \cite{Ram1} and thus, ultimately the strategy of \cite{CarusoLiu}, only adapting its use to the case of Wach modules and the tower of extensions $\{K(\mu_{p^s})\}_s$ rather than $\{K(\pi^{1/{p^s}})\}_s$. Given the similarity in setup, our arguments also resemble those used in \cite[\S 7.3]{EmertonGee2} where the authors adapt the strategy of \cite{CarusoLiu} to the case of certain $\varphi$--modules and the cyclotomic $\ZZ_p$--extension $K_{cyc}/K$, corresponding to the quotient $\mathrm{Gal}(K(\mu_{p^{\infty}})/K)=\ZZ_p^{\times}\twoheadrightarrow \ZZ_p^{\times}/(\ZZ_p^{\times})_{\mathrm{tors}}\simeq \ZZ_p$. 

Let us fix a $p$--torsion Wach module of height $\leq i$, denoted by $M$. In the geometric situation,  we consider $M$ of the form $\overline{\H^i_{\Prism}(\mathscr{X}_p/\AA)}$ or $\H^i_{\Prism}(\mathscr{X}_p/\AA)/p\H^i_{\Prism}(\mathscr{X}_p/\AA)$ for a smooth proper $p$-adic formal scheme $\mathscr{X}/\oh_K$. By Corollary~\ref{cor:WachOnCohomology} and Proposition~\ref{prop:EtaleRealizationGeometric}, these are $p$-torsion Wach modules whose associated Galois representations are $\H^i_{\et}(X_{\CC_K}, \ZZ/p\ZZ)$ and $\H^i_{\et}(X_{\CC_K}, \ZZ_p)/p\H^i_{\et}(X_{\CC_K}, \ZZ_p)$, respectively.

Our aim is to provide a bound on $\mu_{L/K}$ where $L$ is the splitting field of $T(M)$, that is, $L=\overline{K}^{\ker \rho}$ where $\rho: G_K \rightarrow \mathrm{Aut}(T(M))$ is the Galois representation. Noting that this $L$ does not change, we may replace $T(M)$ by its dual, which is related to $M$ by $$T(M)^{\vee}=T^*(M):=\mathrm{Hom}_{\AA, \varphi}(M, \oh_{\CC_K^{\flat}}).$$ 
Moreover, the functor $M\mapsto T^*(M)$ clearly depends on $M$ only up to $(q-1)$--power--torsion; thus, replacing $M$ by its quotient modulo $(q-1)$--power--torsion, we may assume that $M$ is free as a $k[[q-1]]$--module. From now on, let us denote $T^*(M)$ by $T$ for short.

Let us denote by $v^{\flat}$ the tilt of the additive valuation $v_K$ on $\oh_{\mathbb{C}_K}$, i.e. $v^{\flat}(x)=v_K(x^{\sharp})$ where $v_K$ is the valuation on $\oh_{\CC_K}$ determined by $v_K(p)=1$ and where $(-)^{\sharp}: \oh_{\CC_K}^{\flat}\rightarrow \oh_{\CC_K}$ is the multiplicative map $\mathrm{pr}_0: \oh_{\CC_K}^{\flat} = \varprojlim_{x\to x^p} \oh_{\CC_K} \rightarrow \oh_{\CC_K}$. For a real number $c>0$, we denote by $\mathfrak{a}^{>c}$ ($\mathfrak{a}^{\geq c},$ resp.) the ideal of all elements $x \in \oh_{\CC_K}^{\flat}$ with $v^{\flat}(x)>c$ ($v^{\flat}(x)\geq c$, resp.). Clearly every such ideal is stable under the Frobenius map and under the $G_K$--action.

\begin{deff}\label{def:Jc}
For a real number $c>0$, let $J_c$ denote the $G_K$--module
$$J_c=\mathrm{Hom}_{\AA, \varphi}(M, \oh_{\CC_K^{\flat}}/\aa^{>c}).$$ 
The Galois action on $J_c$ is given by the usual formula
$$g(f)(x)=g(f(g^{-1}x)), \;\; g \in G_K, \;\; f \in J_c, \;\; x \in M.$$
We further set $\rho_c:T\rightarrow J_c$ to be the $G_K$--equivariant map induced by the projection $\oh_{\CC_K^{\flat}}\rightarrow \oh_{\CC_K^{\flat}}/\aa^{>c}$. 
Similarly, when $c \geq d >0$, we denote the natural $G_K$--equivariant map $J_c \rightarrow J_d$ by $\rho_{c, d}$, and denote the image of this map by $I_{c, d}$. 
\end{deff}

From now on, let us fix the numbers $$b=\frac{i}{p-1},\;\; a=\frac{pi}{p-1}=pb.$$
The next proposition states that $J_a, J_b$ recover $T$ as a $G_K$--representation completely. It is an analogue of \cite[Proposition~2.3.3]{CarusoLiu} in our context.

\begin{prop}\label{prop:Approximation1}
The map $\rho_b: T \rightarrow J_b$ is injective, and the image agrees with $I_{a, b}$.
\end{prop}

Before proceeding to the proof, let us fix auxiliary data for $M$ and notation that will be useful.

\begin{nott}\label{not:FandV}
Let us fix a free basis $e_1, e_2, \dots e_d$ of $M$. Let $F \in \Mat_{d \times d}(\AA/p)$ be the matrix satisfying $$(\varphi(e_1), \varphi(e_2), \dots, \varphi(e_d))=(e_1, e_2, \dots, e_d)F.$$ Since $M$ is of height $\leq i$, the submodule of $M$ generated by $\varphi(M)$ contains $([p]_q)^iM=(q-1)^{(p-1)i}M$. Thus, there is a matrix $V\in \Mat_{d \times d}(\AA/p)$ such that $FV=(q-1)^{(p-1)i}\mathrm{Id}$. 
\end{nott}

To simplify formulas, we use underlined notation, e.g. $\underline{x}$, to refer to a length $d$ vector $(x_1, x_2, \dots, x_d)$. Thus, for example, $\underline{e}$ refers to the ordered basis $(e_1, e_2, \dots, e_d)$ of $M$. If an operation $f$ makes sense for members of an ordered tuple $\underline{x}$, we use $f(\underline{x})$ to refer to the vector where $f$ is applied term--by--term. Exception to this rule is when $f=v$ is a valuation, in which case $v(\underline{x})$ is a shorthand for $\min_{i} v(x_i)$. Applying this convention to $v=v^{\flat},$ the formula $d(\underline{x}, \underline{y})=p^{-v^{\flat}(\underline{x}-\underline{y})}$ makes $\oh_{\CC_K^{\flat}}^{\oplus d},$ as well as $(\mathfrak{a}^{\geq t})^{\oplus d}$ for every $t \geq 0$, into a complete metric space. 

Recall that a \textit{contraction} on a metric space $X$ is a map $C: X \to X$ with the property that $d(C(x), C(y))\leq \gamma d(x, y)$ for all $x, y \in X$, where $\gamma \in (0, 1)$ is a fixed constant (independent of $x, y$). In particular, a contraction on $(\mathfrak{a}^{\geq t})^{\oplus d}$ is a map $C: (\mathfrak{a}^{\geq t})^{\oplus d}\to (\mathfrak{a}^{\geq t})^{\oplus d}$ such that $v^{\flat}(C(\underline{x})-C(\underline{y}))\geq v^{\flat}(\underline{x}-\underline{y})+k$ for some $k>0$ (independent of $\underline{x}, \underline{y}$).  

\begin{proof}[Proof of Proposition~\ref{prop:Approximation1}]
Clearly $I_{a, b}$ contains $\mathrm{Im}\, \rho_b$, so the second part of the claim amounts to the converse inclusion. Fix $f:M\rightarrow \oh_{\CC_K^{\flat}}/\aa^{>a}$ in $J_a$. We claim that there exists a \textit{unique} $g:M\rightarrow \oh_{\CC_K^{\flat}}$ in $T$ such that $f \equiv g \pmod{\aa^{>b}}$. This shows both that $\rho_b$ is injective and that its image is $I_{a, b}$.

Consider an arbitrary lift $\underline{x}\in (\oh_{\CC}^{\flat})^{\oplus d}$ of the vector $f(\underline{e})$. By compatibility of $f$ with $\varphi$, we have
$$\underline{x}F\equiv \varphi(\underline{x}) \pmod{\aa^{>a}}$$
We may therefore write $\varphi(\underline{x})-\underline{x}F=\underline{r}$ for some $\underline{r} \in (\aa^{\geq c_0})^{\oplus d}$ where $c_0>a$. The existence and uniqueness of $g$ as above now amounts to showing that $\underline{x}'=\underline{x}+\underline{y}$ satisfies $\underline{x'}F= \varphi(\underline{x'})$ for a unique choice of $\underline{y}\in(\aa^{>b})^{\oplus d}$.  
To determine $\underline{y}$, let us compute:
\begin{align*}\underline{x}F+\underline{y}F&=\underline{x'}F=\varphi(\underline{x'})=\varphi(\underline{x})+\varphi(\underline{y})\\
\underline{y}F&=\varphi(\underline{x})-\underline{x}F+\varphi(\underline{y})=\underline{r}+\varphi(\underline{y}).
\end{align*}
Applying $V$ on the left, we arrive at 
\begin{align*}
(\varepsilon - 1)^{\frac{(p-1)i}{p}}\underline{y}&=\underline{r}V+\varphi(\underline{y})V\\
\underline{y}&=(\varepsilon - 1)^{-\frac{(p-1)i}{p}}(\underline{r}V+\varphi(\underline{y})V). \stepcounter{equation}\tag{\theequation}\label{eqn:FixedPt}
\end{align*}
We aim to show that the map $C:\underline{y}\mapsto (\varepsilon - 1)^{-\frac{(p-1)i}{p}}(\underline{r}V+\varphi(\underline{y})V)$ has a unique fixed point $\underline{y}$ in $(\aa^{>b})^{\oplus d}$. Consider any $c\in (a, c_0)$, so that $\underline{r} \in (\mathfrak{a}^{\geq c})^{\oplus d}$. Then it is easy to check that $C$ takes $(\aa^{\geq(c/p)})^{\oplus d}$ to $(\aa^{\geq(c/p)})^{\oplus d}$. Indeed, assuming $\underline{y} \in (\aa^{\geq(c/p)})^{\oplus d},$ we have
$$v^{\flat}(C(\underline{y}))\geq \mathrm{min}\left(v^{\flat}(\underline{r}), v^{\flat}(\varphi(\underline{y}))\right) +v^{\flat}\left((\varepsilon - 1)^{-\frac{(p-1)i}{p}}\right)\geq c-a\cdot\frac{p-1}{p}\geq c\left(1-\frac{p-1}{p}\right)=\frac{c}{p}.$$ Moreover, $C$ is a contraction on $(\aa^{\geq(c/p)})^{\oplus d}$: if $v^{\flat}(\underline{y}_1-\underline{y}_2) \geq t\geq c/p$, then 
$$v^{\flat}(C(\underline{y}_1)-C(\underline{y}_2))=v^{\flat}((\varepsilon - 1)^{-\frac{(p-1)i}{p}}\varphi(\underline{y_1}-\underline{y_2})V)\geq pt-i \geq t+h(p-1)$$ 
where $h=(c-a)/p>0$. By the Banach fixed-point theorem \cite[Th\'{e}or\`{e}me 6]{Banach}, there is a unique $\underline{y} \in (\mathfrak{a}^{\geq c/p})^{\oplus d}$ fixed by $C$. Moreover, changing $c$ to $c'$ with $c>c'>a$ does not change $\underline{y}$ since $(\mathfrak{a}^{\geq c/p})^{\oplus d} \subseteq (\mathfrak{a}^{\geq c'/p})^{\oplus d}$. Since $(\mathfrak{a}^{>b})^{\oplus d}$ is the directed union of $(\mathfrak{a}^{\geq c/p})^{\oplus d}$ over all $c>a$ as above, it follows that this $\underline{y}$ is the unique fixed point of $C$ on $(\mathfrak{a}^{>b})^{\oplus d}$. Thus, (\ref{eqn:FixedPt}) has a unique solution in $(\mathfrak{a}^{>b})^{\oplus d}$, which finishes the proof.
\end{proof}

For $c>0$ and an integer $s \geq 0$, we say that the action on $J_c$ is \textit{$G_s$--formal} if for all $g \in G_s$, $f \in J_c$ and all $x \in M,$ we have $g(f)(x)=g(f(x))$. Equivalently, $f:M\rightarrow \oh_{\CC_K^{\flat}}/\aa^{>c}$ is invariant for the action of $G_s$ on the source (``formal'' here refers to the fact that one may disregard the action of $G_s$ on $M$ and still get the correct action on $J_c$). The following result is crucial for establishing the bounds.

\begin{prop}\label{prop:GsFormal}
The action on $J_c$ is $G_s$--formal when $p^{s-\delta_p}>c(p-1)$. In particular, the action on $J_b$ is $G_s$--formal when $p^{s-\delta_p}>i$.
\end{prop}

\begin{proof}
Let $f \in J_c$ be arbitrary. By Proposition~\ref{prop:CrystallineConditionAnalogue}, for every $x \in M$ and $g \in G_s$, we have $g(x)-x=(q-1)^{p^{s-\delta_p}}y$ for some $y$. Applying $f$, we obtain $f(g(x))-f(x)=(\varepsilon^{1/p}-1)^{p^{s-\delta_p}}f(y)$. Since $v^{\flat}(\varepsilon^{1/p}-1)=1/(p-1),$ we infer that $(\varepsilon^{1/p}-1)^{p^{s-\delta_p}}f(y)=0$ in $\oh_{\CC_K^{\flat}}/\aa^{>c})$ when $p^{s-\delta_p}/(p-1)>c$. That is, under the assumption $p^{s-\delta_p}>c(p-1)$, every $f \in J_c$ is $G_s$-invariant, as desired.
\end{proof}

We further need a version of Definition~\ref{def:Jc} with restricted coefficients.

\begin{deff}
Fix an integer $s \geq 0$ and a real number $c$ satisfying $p^s>c>0$. For an algebraic extension $E/K_{s},$ define
$$J_c^{(s)}(E)=\mathrm{Hom}_{\varphi, \AA}(M, (\varphi_k^{s})^*\oh_E/\aa_E^{>c/p^s}).$$
Here $\varphi_k^s$ denotes the $s$--th power of the Frobenius of $k$. Since $1>c/p^s$, $p=0$ in $\oh_E/\aa_E^{>c/p^s}$ and it is therefore naturally a $k$-algebra, so the indicated pullback makes sense. We further view $(\varphi_k^{s}))^*\oh_E/\aa_E^{>c/p^s}$ as an $\AA/p$--module via $q-1\mapsto \zeta_{p^{s+1}}-1$. When the extension $E/K_{s}$ is Galois, we endow $J_c^{(s),E}$ with the action of $G_{s}$, given by $$g(f)(x)=g(f(g^{-1}x)),\;\; g \in G_{s},\; f \in J_c^{(s),E}, \; x \in M.$$
\end{deff}

\begin{rem}\label{rem:TwistedAction}
When $E=\overline{K}$, there is a $G_{s}$--equivariant isomorphism $\oh_{\CC_K^{\flat}}/\aa^{>c} \simeq (\varphi_k^{s})^*\oh_{\overline{K}}/\aa_{\overline{K}}^{>c/p^s}$; consequently, there is an induced isomorphism $J_c \simeq J_c^{(s)}(\overline{K})$ of  $G_{s}$--modules. Similarly, when $F/E/K_{s}$ is a tower of algebraic extensions, the map $\oh_{E}/\aa_E^{>c/p^s} \rightarrow \oh_{F}/\aa_F^{>c/p^s}$ is injective (note that $\aa_F^{>c/p^s}\cap \oh_E=\aa_E^{>c/p^s}$) and it is $G_{s}$--equivariant when both $F$ and $E$ are Galois over $K_{s}$. Thus, we obtain an injection $J_c^{(s)}(E) \rightarrow J_c^{(s)}(F)$, which is  $G_{s}$--equivariant in the Galois case. 

Fixing $E$ and $s$, for $0<d \leq c < p^s$ we have the evident map $J_c^{(s)}(E)\rightarrow J_d^{(s)}(E)$ induced by the quotient map $\oh_{\overline{K}}/\aa_{\overline{K}}^{>c/p^s}\rightarrow \oh_{\overline{K}}/\aa_{\overline{K}}^{>d/p^s}$. Denote this map by $\rho_{c, d}^{(s)}(E)$, and its image by $I_{c, d}^{(s)}(E)$.
\end{rem}

Finally, we introduce a variant where we lift the coefficients to $\oh_{E}$ from $\oh_{E}/\aa_{E}^{>c/p^s}$. For that purpose, we fix some further notation first.  

\begin{nott}\label{not:FandVLifts}
For an integer $s$ with $p^s>i$, let $F_{(s)}, V_{(s)}$ be the images of the matrices $F$ and $V$, resp., under the map $k[[q-1]]\rightarrow (\varphi_k^s)^*\oh_{K_{s}}/p$. We then identify $(\varphi_k^s)^*\oh_{K_{s}}/p$ with $\oh_{K_{s}}/p$ and consider some lifts $\widetilde{F}_{(s)}, \widetilde{V}'_{(s)}$ of $F_{(s)}$ and $V_{(s)},$ resp., to $\oh_{K_{s}}$. Then we have 
$$\widetilde{F}_{(s)}\widetilde{V}'_{(s)}\equiv (\zeta_{p^{s+1}}-1)^{(p-1)i}\mathrm{Id} \pmod{p}.$$
It follows that $\widetilde{F}_{(s)}\widetilde{V}'_{(s)}=(\zeta_{p^{s+1}}-1)^{(p-1)i}(\mathrm{Id}+C)$ for a matrix $C$ with entries in $\aa_{K_{s}}^{>0}$ (here we use that $i/p^s <1=v_{K}(p)$). The matrix $\mathrm{Id}+C$ has an inverse of the form $\mathrm{Id}+D$ where $D$ again has entries in $\aa_{K_{s}}^{>0}$ (it is given by $\sum_{n=1}^{\infty}(-C)^n$). Set $\widetilde{V}_{(s)}=\widetilde{V}'_{(s)}(\mathrm{Id}+D)$. The resulting matrices then satisfy the identity (of matrices over $K_{s}$)
$$\widetilde{F}_{(s)}\widetilde{V}_{(s)}=(\zeta_{p^{s+1}}-1)^{(p-1)i}\mathrm{Id}.$$ 
\end{nott} 

\begin{deff}
Given an integer $s$ with $p^s>i$ and an algebraic extension $E/K_{s},$ the set $\widetilde{J}^{(s)}(E)$ is defined as
$$\widetilde{J}^{(s)}(E)=\{\underline{\widetilde{x}}=(\widetilde{x}_1, \widetilde{x}_2, \dots, \widetilde{x}_d)\in \oh_E^{\oplus d}\;|\; \underline{\widetilde{x}^p}=\underline{\widetilde{x}}\widetilde{F}_{(s)}\},$$ 
where $\underline{\widetilde{x}^p}$ denotes the vector $(\widetilde{x}_1^p, \widetilde{x}_2^p, \dots, \widetilde{x}_d^p)$.
When $E/K_{s},$ is Galois, we endow $\widetilde{J}^{(s)}(E)$ with action on entries of $\underline{\widetilde{x}}$.
\end{deff}

Given $c$ with $0<c<p^s$, there is an evident map $\widetilde{\rho}^{(s)}_c(E):\widetilde{J}^{(s)}(E) \rightarrow J_c^{(s)}(E)$. When $E=\overline{K}$, this map is $G_{s}$--equivariant as long as $J_c$ is $G_{s}$--formal. By Proposition~\ref{prop:GsFormal}, this is indeed the case as long as $p^{s-\delta_p}/(p-1)>c$. In particular, when $p$ is odd, this is true for $c:=b=i/(p-1)$ under our assumption $p^s>i$, and when $p=2$, we need the stronger assumption $2^s>2i$. Note that both these conditions are satisfied under the slightly stronger condition $(p-1)p^{s-1}>i$, or equivalently, $p^s>a$, which we assume from now on.

We need the following enhancement of Proposition~\ref{prop:Approximation1}, analogous to \cite[Lemma~4.1.4]{CarusoLiu}.

\begin{prop}\label{prop:Approximation2}
Assume $p^s>a,$ that is, $(p-1)p^{s-1}>i$. Then for every algebraic extension $E/K_{s},$ the map $\widetilde{\rho}^{(s)}_b(E):\widetilde{J}^{(s)}(E) \rightarrow J_b^{(s)}(E)$ is injective and its image is $I^{(s)}_{a, b}(E)$.
\end{prop}

\begin{proof} We proceed as in the proof of Proposition~\ref{prop:Approximation1}. Let us identify all the rings of the form $(\varphi_k^s)^*(\oh_E/\aa^{>t})$ appearing in the proof with $\oh_E/\aa^{>t}$. Fix $f \in J_b^{(s)}(E)$ and let $\underline{f}(\underline{e})\in (\oh_{E}/\aa^{>a/p^s})^{\oplus d}$ be the vector of images of the fixed basis $\underline{e}$ from Notation~\ref{not:FandV}. Choosing any lift $\underline{x}$ of $\underline{f}(\underline{e})$ to $\oh_{E}^{\oplus d}$, the aim is to show that there is a unique $\underline{y}\in (\aa^{>b/p^s})^{\oplus d}$ such that $\underline{x}+\underline{y} \in \widetilde{J}^{(s)}(E)$. 

The required equation then takes the form
$$(\underline{x}+\underline{y})^p=(\underline{x}+\underline{y})\widetilde{F}_{(s)}$$
which, after applying $\widetilde{V}_{(s)}$ on the right and simplifying, becomes
$$\underline{y}=(\zeta_{p^{s+1}}-1)^{-i(p-1)}(\underline{x}+\underline{y})^p\widetilde{V}_{(s)}-\underline{x}$$
(recall that we use the notation for $p$--th power of vectors component--by--component). Just like in the proof of Proposition~\ref{prop:Approximation1}, the goal is to show that the map $C$ given by
$$C(\underline{y})=(\zeta_{p^{s+1}}-1)^{-i(p-1)}(\underline{x}+\underline{y})^p\widetilde{V}_{(s)}-\underline{x}$$
takes $(\aa_E^{\geq c})^{\oplus d}$ to $(\aa_E^{\geq c})^{\oplus d}$ for $c$ with $c>b$ (and arbitrarily close to $b$), and that $C$ is a contraction on this space. 
To choose such $c$, note that we have $\underline{x}^p \equiv \underline{x}\widetilde{F}_{(s)}, \pmod{\aa_E^{>a/p^s}},$ and after applying $\widetilde{V}_{(s)}$ on the right, it follows that   $(\zeta_{p^{s+1}}-1)^{-i(p-1)}\underline{x}^p \widetilde{V}_{(s)} \equiv \underline{x} \pmod{\aa_E^{>t}}$ for $t=a/p^s-i/p^s=b/p^s.$ Then we may choose $c$ so that this congruence still holds modulo $\aa_E^{\geq c/p^s}$.

The fact that $C$ takes $(\aa_E^{\geq {c/p^s}})^{\oplus d}$ to $(\aa_E^{\geq {c/p^s}})^{\oplus d}$ is then easily seen as follows. Write
$$C(\underline{y})=\underbrace{(\zeta_{p^{s+1}}-1)^{-i(p-1)}\underline{x}^p\widetilde{V}_{(s)}-\underline{x}}_\alpha+\underbrace{(\zeta_{p^{s+1}}-1)^{-i(p-1)}p\underline{y}R\widetilde{V}_{(s)}}_\beta+\underbrace{(\zeta_{p^{s+1}}-1)^{-i(p-1)}\underline{y}^p\widetilde{V}_{(s)}}_\gamma,$$
where $p\underline{y}R$ consists of the mixed terms from the binomial expansion of $(\underline{x}+\underline{y})^p$. Then $\alpha \in (\aa_E^{\geq c/p^s})^{\oplus d}$ by the choice of $c$, and we further have $\beta \in (\aa_E^{\geq t})^{\oplus d}$ for $t\geq 1+c/p^s-i/p^s$ and $\gamma \in (\aa_E^{\geq u})^{\oplus d}$ for $u \geq cp/p^s-i/p^s,$ both of which are bigger than $c/p^s.$

To show that $C$ is a contraction on $\aa_E^{\geq c/p^s}$, let $h_0=\min\{1, (p-1)c/p^s\}$ and $h=h_0-i/p^s$ (then $h>0$). For $\underline{y}_1, \underline{y}_2 \in \aa_E^{\geq c/p^s}$ with $v_K(\underline{y}_1-\underline{y}_2)=t,$ we have
\begin{align*}
C(\underline{y}_1)-C(\underline{y}_2)&=(\zeta_{p^{s+1}}-1)^{-i(p-1)}((\underline{x}+\underline{y}_1)^p-(\underline{x}+\underline{y}_2)^p)\widetilde{V}_{(s)}\\
&=(\zeta_{p^{s+1}}-1)^{-i(p-1)}(\underbrace{p(\underline{y}_1-\underline{y}_2)S}_\delta+\underbrace{(\underline{y}_1^p-\underline{y}_2^p)}_{\epsilon})\widetilde{V}_{(s)}
\end{align*} 
The term $\delta$ consists of all the mixed terms in binomial expansions of $(\underline{x}+\underline{y}_1)^p$ and $(\underline{x}+\underline{y}_2)^p$ (and it is easy to see that it has the indicated form, with $S$ an integral matrix). The valuation of $\delta$ is therefore at least $1+t$, hence $1+t-i/p^s\geq t+h$ after multiplying by $(\zeta_{p^{s+1}}-1)^{-i(p-1)}$. Regarding $\epsilon$, we have
$$\underline{y}_1^p-\underline{y}_2^p=(\underline{y}_1-\underline{y}_2)(\underline{y}_1^{p-1}+\underline{y}_1^{p-2}\underline{y}_2+ \dots + \underline{y}_1\underline{y}_2^{p-2}+\underline{y}_2^{p-1})\,,$$ and therefore $\epsilon$ is of valuation at least $t+(p-1)c/p^s$. After accounting for $(\zeta_{p^{s+1}}-1)^{-i(p-1)}$, this becomes $t+(p-1)c/p^s-i/p^s \geq t+h$. Thus, the valuation of $C(\underline{y}_1)-C(\underline{y}_2)$ is at least $t+h$, showing that $C$ is a contraction on $(\aa_E^{\geq c/p^s})^{\oplus d}$ and thus, finishing the proof.
\end{proof}

Let us denote $L_{s}=LK_{s}=L[\zeta_{p^{s+1}}]$. Consider $s$ with $p^s>a,$ and an algebraic extension $E/K_{{s}}$. The four canonical maps between $J_{a}^{(s)}(E), J_{b}^{(s)}(E), J_a$ and $J_b$ induce the inclusion
$$\iota_{E, s}: I_{a, b}^{(s)}(E)=\rho_{a, b}^{(s)}(E)(J_a^{(s)}(E)) \hookrightarrow \rho_{a, b}(J_a)=I_{a, b}.$$

The next proposition regarding $\iota_{E, s}$ is a basis for establishing validity of Fontaine's property $(P_m)$ in our context; it is a direct analogue of \cite[Theorem~5.13]{Ram1} and \cite[Theorem~4.1.1]{CarusoLiu}.

\begin{prop}\label{prop:FixedPoints}
The map $\iota_{E, s}$ is an isomorphism if and only if $L_{s}\subseteq E$.
\end{prop}  

\begin{proof}
We have a series of $G_{s}$--equivariant bijections
$$\widetilde{J}^{(s)}(\overline{K}) \simeq I_{a, b}^{(s)}(\overline{K}) \simeq I_{a, b} \simeq T,$$
where the indicated isomorphisms use Proposition~\ref{prop:Approximation2}, Remark~\ref{rem:TwistedAction}, and Proposition~\ref{prop:Approximation1}, respectively.
Similarly, by Proposition~\ref{prop:Approximation2} we have a $G_{s}$--equivariant isomorphism $\widetilde{J}^{(s)}(E) \simeq I_{a, b}^{(s)}(E)$, and we clearly have $\widetilde{J}^{(s)}(E)=\widetilde{J}^{(s)}(\overline{K})^{G_E}$. Thus, the map $\iota_{E, s}$ may be replaced by the inclusion $T^{G_E}\subseteq T$, for which the statement of the proposition is obviously valid.
\end{proof}

\begin{prop}\label{prop:ProofOfPm}
Let $s$ be an integer such that $p^s>a$, and let $m=a/p^s$. Then Fontaine's property $(P_m^{L_{s}/K_{s}})$ holds.
\end{prop}

\begin{proof}
We follow the proof of \cite[Proposition~5.14]{Ram1}, ultimately based on the arguments of \cite{Hattori, CarusoLiu}. By Proposition~\ref{prop:PmAndRamification}, we may replace $K_{s}$ by the maximal unramified extension $K_{s}^{\mathrm{un}}$ inside $L_{s}$, and prove $(P_m^{L_{s}/K_{s}^{\mathrm{un}}})$ instead. 

Let $E/K_{s}^{\mathrm{un}}$ be an algebraic extension and let $f: \oh_{L_{s}}\rightarrow \oh_{E}/\aa_E^{>m}$ be an $\oh_{K_{s}^{\mathrm{un}}}$--algebra map. For $c \in \{a, b\},$ we consider the induced map
$$f_c: \oh_{L_{s}}/\aa_{L_{s}}^{>c/p^s}\rightarrow \oh_{E}/\aa_{E}^{>c/p^s}.$$  
First, we claim that this map is well--defined an injective.
To prove this, consider a uniformizer $\varpi \in L_{s}$. The extension $L_{s}/K_{s}^{\mathrm{un}}$ is totally ramified, so $\varpi$ satisfies an Eisenstein relation of the form (with $e=e(L_{s}/K_{s})$)
$$\varpi^e=c_1\varpi^{e-1}+c_2\varpi^{e-2}+\dots +c_{e-1}\varpi+c_e,$$ 
with $v_K(c_i)\geq 1/(p^s(p-1))$ for all $i$, and $v_K(c_e)=1/(p^s(p-1))$. Applying $f,$ the same relation applies to $f=f(\varpi) \in \oh_{E}/\aa_E^{>m}$. Choosing a lift $\widetilde{t}$ of $t$ to $\oh_E$, we then obtain the relation 
$$\widetilde{t}^e=c_1\widetilde{t}^{e-1}+c_2\widetilde{t}^{e-2}+\dots +c_{e-1}\widetilde{t}+c_e+r,$$ 
where $r \in \aa_E^{>m}$. Since $m\geq 1/(p^s(p-1)),$ the valuation of the left--hand side is that of $c_e$, and it follows that $v_K(\widetilde{t})=1/(ep^s(p-1))=v_K(\varpi).$ We may therefore conclude that
$$\forall N: \;\; \varpi^N \in \aa_{L_{s}}^{>c/p^s} \text{ if and only if } \frac{N}{ep^s(p-1)}>\frac{c}{p^s} \text{ if and only if } \widetilde{t}^N \in \aa_{E}^{>c/p^s}.$$
The `only if' part shows that $f_c$ is well--defined, and the `if' part shows that it is injective.

Applying $(\varphi_k^s)^*(-)$ to $f_a$ and $f_b$, one obtains a commutative square 

\begin{center}
\begin{tikzcd}[column sep = large]
(\varphi_k^s)^*\oh_{L_{s}}/\aa_{L_{s}}^{>a/p^s} \ar[r, hook, "(\varphi_k^s)^*(f_a)"] \ar[d] & (\varphi_k^s)^*\oh_{E}/\aa_{E}^{>a/p^s} \ar[d] \\
(\varphi_k^s)^*\oh_{L_{s}}/\aa_{L_{s}}^{>b/p^s} \ar[r, hook, "(\varphi_k^s)^*(f_b)"] & (\varphi_k^s)^*\oh_{E}/\aa_{E}^{>b/p^s}\;,
\end{tikzcd}
\end{center}
which in turn induces a commutative square
\begin{center}
\begin{tikzcd}[column sep = large]
J_a^{(s)}(L_{s}) \ar[r, hook] \ar[d, "\rho_{a,b}^{(s)}(L_{s})"] & J_a^{(s)}(E) \ar[d, "\rho_{a,b}^{(s)}(E)"] \\
J_b^{(s)}(L_{s}) \ar[r, hook] & J_b^{(s)}(E)\;.
\end{tikzcd}
\end{center}
Taking images of the vertical maps, we obtain an injection $I_{a, b}^{{(s)}}(L_{s}) \hookrightarrow I_{a, b}^{{(s)}}(E)$. By Proposition~\ref{prop:FixedPoints} applied to $L_{s}$, $I_{a, b}^{{(s)}}(L_{s}) \simeq T$, and therefore the injection $\iota_{E, s}$ from Proposition~\ref{prop:FixedPoints} has to be an isomorphism (it is an injection of finite sets where the source has size at least as much as the target). It follows by Proposition~\ref{prop:FixedPoints} again that $L_{s}\subseteq E$. This finishes the proof.
\end{proof}

Finally, we are ready to prove the ramification bound in full.

\begin{thm}\label{thm:conclusion}
For $$\alpha=\left(\left\lfloor \log_p\left(\frac{ip}{p-1}\right) \right\rfloor+1\right)\;\;\text{and}\;\;\beta=\mathrm{max}\left\{0, \frac{ip}{p^{\alpha}(p-1)}-\frac{1}{p-1}\right\},$$ 
one has $\mu_{L/K}\leq 1+\alpha+\beta.$
\end{thm}

\begin{proof}
Set $s=\alpha$ so that $p^s>ip/(p-1)=a$. 
First we estimate $\mu_{L_{s}/K}.$ By Lemma~\ref{lem:MuInTowers}, we have
$$\mu_{L_{s}/K}=\mathrm{max}\{\mu_{K_{s}/K}, \phi_{K_{s}/K}(\mu_{L_{s}/K_{s}})\}.$$
Propositions~\ref{prop:ProofOfPm} and \ref{prop:PmAndRamification} show that $\mu_{L_{s}/K_{s}} \leq p^s(p-1)m=ip$. A classical computation (e.g. \cite[\S IV]{SerreLocalFields}) shows that $\mu_{K_{s}/K}=s+1,$ and that $\phi_{K_{s}/K}(t)$ has the last break point given by $\phi_{K_{s}/K}(p^s)=s+1$, with last slope $1/(p^s(p-1))$. Therefore, we may estimate $$\phi_{K_s/K}(t)\leq 1+s-\frac{1}{p-1}+\frac{t}{p^s(p-1)},$$ and we obtain
$$\mu_{L_{s}/K}\leq \mathrm{max}\{1+s, 1+s-\frac{1}{p-1}+\frac{ip}{p^s(p-1)}\}=1+s+\mathrm{max}\{0, \frac{ip}{p^s(p-1)}-\frac{1}{p-1}\}.$$
Finally, observing that $\mu_{L/K}\leq \mu_{L_{s}/K}$, we obtain the desired bound.
\end{proof}

\begin{rem}\label{rem:CrystallineVsSemistable}
Let us briefly compare the bound from Theorem~\ref{thm:conclusion} with the results \cite{Hattori, CarusoLiu} (i.e., the semistable case) and \cite{Ram1}. For general comparison between the three, see \cite[\S~5.2]{Ram1}. Here we only summarize that for $K$ absolutely unramified, the bounds of \cite{CarusoLiu, Ram1} both become 
\begin{equation}\label{CommonSemistableBound}
\mu_{L/K} \leq 1+\alpha+\mathrm{max}\left\{\frac{ip}{p^{\alpha}(p-1)}-\frac{1}{p^\alpha}, \frac{1}{p-1}\right\},
\end{equation}
(with the same $\alpha$ as in Theorem~\ref{thm:conclusion}). In order to compare with \cite{Hattori}, one needs to further assume $i<p-1$, in which case all three bounds agree. 

On the other hand, Theorem~\ref{thm:conclusion} gives 
\begin{equation}
\mu_{L/K} \leq 1+\alpha+\mathrm{max}\left\{\frac{ip}{p^{\alpha}(p-1)}-\frac{1}{p-1}, \; 0 \,\right\},
\end{equation}
which is a stronger bound in all cases (since $\alpha$ is always at least $1$).  
\end{rem}

\begin{pr}
To show that Theorem~\ref{thm:conclusion} in general excludes torsion \textit{semistable} representations, we consider the example from \cite{Hattori} for which the ramification bounds of \textit{loc. cit.} are optimal. 
Let $K=\QQ_p$, and consider the Tate curve $E_p$ at $p$, i.e., the elliptic curve over $\QQ_p$ with $E_P(\overline{\QQ_p})=\overline{\QQ_p}^{\times}/p^{\ZZ}$. It is well--known that $E_p$ has semistable (and not good) reduction. 

The set $E_p(\overline{\QQ_p})[p]$ can be identified with the set of elements of the form $\zeta_p^i p^{j/p}$ where $0\leq i,j \leq p-1.$ Consequently, the splitting field for $\H^1_{\et}(E_{p, \CC_K}, \ZZ/p\ZZ)$ is $L=\QQ_p(\zeta_p, p^{1/p})$. By \cite[Remark~5.5]{Hattori}, one has $\mu_{L/K}=2+1/(p-1)$. 

On the other hand, the bound from Theorem~\ref{thm:conclusion} for $i=1$ would give the stronger estimate $\mu_{L/K}\leq 2$. This shows that the bounds obtained in Theorem~\ref{thm:conclusion} are ``genuinely crystalline'', i.e. not satisfied by varieties with semistable reduction in general.
\end{pr}

\bibliography{references}

@article{Abrashkin,
  title={Ramification in {\'e}tale cohomology},
  author={Abrashkin, Victor},
  journal={Invent. Math.},
  volume={101},
  number={1},
  pages={631--640},
  year={1990},
  publisher={Springer},
  zblnumber={0761.14006},
  mrnumber={1062798}
}

@article{Abrashkin2,
  title={Ramification estimate for {F}ontaine--{L}affaille {G}alois modules},
  author={Abrashkin, Victor},
  journal={J. Algebra},
  volume={427},
  pages={319--328},
  year={2015},
  publisher={Elsevier},
  zblnumber={1322.11122},
  mrnumber={3312307}
}

@article{BMS1,
  title={Integral {$ p $}--adic {H}odge theory},
  author={Bhatt, Bhargav and Morrow, Matthew and Scholze, Peter},
  journal={Publ. Math., Inst. Hautes {\'E}tud. Sci.},
  volume={128},
  number={1},
  pages={219--397},
  year={2018},
  publisher={Springer},
  zblnumber={1446.14011},
  mrnumber={3905467}

}

@article{BMS2,
  title={Topological {H}ochschild homology and integral {$ p $}--adic {H}odge theory},
  author={Bhatt, Bhargav and Morrow, Matthew and Scholze, Peter},
  journal={Publ. Math., Inst. Hautes {\'E}tud. Sci.},
  volume={129},
  number={1},
  pages={199--310},
  year={2019},
  publisher={Springer},
  zblnumber={1478.14039},
  mrnumber={3949030}
}

@article{BhattScholze,
  title={Prisms and prismatic cohomology},
  author={Bhatt, Bhargav and Scholze, Peter},
  journal={Ann. Math. (2)},
  volume={196},
  number={3},
  pages={1135--1275},
  year={2022},
  publisher={Department of Mathematics of Princeton University},   
  zblnumber={1552.14012},
  mrnumber={4502597}
}

@article{Caruso,
  title={Repr{\'e}sentations galoisiennes {$ p $}-adiques et {$(\varphi,\tau)$}--modules},
  author={Caruso, Xavier},
  journal={Duke Math. J.},
  volume={162},
  number={13},
  pages={2525--2607},
  year={2013},
  publisher={Duke University Press},
  zblnumber={1294.11207},
  mrnumber={3127808}
}

@article{CarusoLogCris,
  title={Conjecture de l’inertie mod{\'e}r{\'e}e de {S}erre},
  author={Caruso, Xavier},
  journal={Invent. Math.},
  volume={171},
  number={3},
  pages={629--699},
  year={2008},
  publisher={Springer},
  zblnumber={1245.14019},
  mrnumber={2372809}
}

@article{CarusoLiu,
  title={Some bounds for ramification of {$p^n$}--torsion semi--stable representations},
  author={Caruso, Xavier and Liu, Tong},
  journal={J. Algebra},
  volume={325},
  number={1},
  pages={70--96},
  year={2011},
  publisher={Elsevier},
  zblnumber={1269.14001},
  mrnumber={2745530}
}

@article{EmertonGee1,
  title={{$p$}-adic {H}odge-theoretic properties of {\'e}tale cohomology with mod {$p$} coefficients, and the cohomology of {S}himura varieties},
  author={Emerton, Matthew and Gee, Toby},
  journal={Algebra Number Theory},
  volume={9},
  number={5},
  pages={1035--1088},
  year={2015},
  publisher={Mathematical Sciences Publishers},
  zblnumber={1321.11050},
  mrnumber={3365999}
}

@book{EmertonGee2,
url = {https://doi.org/10.1515/9780691241364},
title = {Moduli stacks of \'{e}tale {$(\varphi, \Gamma)$}--modules and the existence of crystalline lifts},
author = {Matthew Emerton and Toby Gee},
publisher = {Princeton University Press},
series = {Ann. Math. Stud.},
address = {Princeton},
doi = {doi:10.1515/9780691241364},
isbn = {9780691241364},
year = {2023},
  zblnumber={1529.11003},
  mrnumber={4529886}
}

@article{Fontaine,
  title={Il n'y a pas de vari{\'e}t{\'e} ab{\'e}lienne sur {$\mathbb{Z}$}},
  author={Fontaine, Jean-Marc},
  journal={Invent. Math.},
  volume={81},
  number={3},
  pages={515--538},
  year={1985},
  publisher={Springer},
  zblnumber={0612.14043},
  mrnumber={0807070}
}

@Inproceedings{Fontaine2,
  author       = {Fontaine, Jean-Marc},
  title        = {Sch\'{e}mas propres et lisses sur {$\mathbb{Z}$}},
  booktitle    = {Proceedings of the {I}ndo--{F}rench {C}onference on {G}eometry},
  year         = {1993},
  pages        = {43--56},
  address      = {New Delhi},
  publisher    = {Hindustan Book Agency},
  zblnumber={0837.14014},
  mrnumber={1274493}
}

@inproceedings{FontaineLaffaille,
  title={Construction de repr{\'e}sentations {$p$}-adiques},
  author={Fontaine, Jean-Marc and Laffaille, Guy},
  booktitle={Ann. Sci. {\'E}c. {N}orm. Sup{\'e}r.},
  volume={15},
  number={4},
  pages={547--608},
  year={1982},
  zblnumber={0579.14037},
  mrnumber={0707328}
}

@InProceedings{FontaineMessing,
  author       = {Fontaine, Jean-Marc and Messing, William},
  title        = {{$p$}-adic periods and {$p$}-adic {\'e}tale cohomology},
  booktitle    = {Current trends in arithmetical algebraic geometry},
  year         = {1987},
  volume       = {67},
  series       = {Contemp. Math.},
  pages        = {179--207},
  address      = {Arcata, Calif., 1985},
  organization = {Amer. Math. Soc.},
  zblnumber={0632.14016},
  mrnumber={0902593}
}

@article{Hattori,
  title={On a ramification bound of torsion semi-stable representations over a local field},
  author={Hattori, Shin},
  journal={J. Number Theory},
  volume={129},
  number={10},
  pages={2474--2503},
  year={2009},
  publisher={Elsevier},
  zblnumber={1205.11127},
  mrnumber={2541026}
}

@article{LiLiu,
  title={Comparison of prismatic cohomology and derived de {R}ham cohomology},
  author={Li, Shizhang and Liu, Tong},
  journal={J. {E}ur. {M}ath. {S}oc. (JEMS)},
  volume={27},
  number={1},
  pages={183–-268},
  year={2025},
  zblnumber={7981649},
  mrnumber={4859568}
}

@book{SerreLocalFields,
  title={Local fields},
  author={Serre, Jean-Pierre},
  fseries = {Graduate Texts in Mathematics},
  series = {Grad. Texts Math.},
  note = {Translated from the {F}rench by {M}arvin {J}ay {G}reenberg},
  volume={67},
  year={1979},
  publisher={Springer},
  zblnumber={0423.12016},
  mrnumber={554237}
}

@misc{stacks,
shorthand = {Stacks},
author = {{Stacks Project Authors}, The},
title = {Stacks Project},
howpublished = {available online at~\url{http://stacks.math.columbia.edu}},
year = {2026},
}

@article{Yoshida,
  title={Ramification of local fields and {F}ontaine's property {$({P}_m)$}},
  author={Yoshida, Manabu},
  journal={J. Math. Sci., Tokyo},
  volume={17},
  number={3},
  pages={247--265},
  year={2010},
  publisher={Department of Mathematical Sciences, the University of Tokyo},
  zblnumber={1263.11105},
  mrnumber={2814776}
}

@article{BhattScholzeCrystals,
  title={Prismatic $ F $-crystals and crystalline Galois representations},
  author={Bhatt, Bhargav and Scholze, Peter},
  journal={Cambridge Journal of Mathematics},
  volume={11},
  number={2},
  pages={507--562},
  year={2023},
  publisher={International Press of Boston}
}

@article{BreuilMessing,
  title={Torsion {\'e}tale and crystalline cohomologies},
  author={Breuil, Christophe and Messing, William},
  journal={Ast{\'e}risque},
  volume={279},
  pages={81--124},
  year={2002}
}

@misc{BreuilLetter,
  author       = "Breuil, Christophe",
  title        = "Letter to {G}ross",
  year         = "1998",
}

@article{LiuPhiG,
  title={A note on lattices in semi-stable representations},
  author={Liu, Tong},
  journal={Mathematische Annalen},
  volume={346},
  number={1},
  pages={117--138},
  year={2010},
  publisher={Springer}
}

@incollection{Kisin,
  title={Crystalline representations and {F}-crystals},
  author={Kisin, Mark},
  booktitle={Algebraic geometry and number theory: in honor of {V}ladimir {D}rinfeld’s 50th birthday},
  pages={459--496},
  year={2006},
  publisher={Springer}
}

@article{Ram1,
 author = {{\v{C}}oupek, Pavel},
 title = {Crystalline condition for {{\(A_{\text{inf}}\)}}-cohomology and ramification bounds},
 fjournal = {Documenta Mathematica},
 journal = {Doc. Math.},
 issn = {1431-0635},
 volume = {31},
 number = {1},
 pages = {141--195},
 year = {2026}
}

@article{Wach,
  title={Repr{\'e}sentations $ p $-adiques potentiellement cristallines},
  author={Wach, Nathalie},
  journal={Bull. de la Soc. Math. de France},
  volume={124},
  number={3},
  pages={375--400},
  year={1996}
}

@article{Wach2,
  title={Repr{\'e}sentations cristallines de torsion},
  author={Wach, Nathalie},
  journal={Compositio Mathematica},
  volume={108},
  number={2},
  pages={185--240},
  year={1997},
  publisher={London Mathematical Society}
}

@article{ColmezWachM,
  title={Repr{\'e}sentations cristallines et repr{\'e}sentations de hauteur finie},
  author={Colmez, Pierre},
  journal={Journal fur die Reine und Angewandte Mathematik},
  volume={514},
  pages={119--144},
  year={1999},
  publisher={Berlin, W. de Gruyter.}
}

@article{BergerWachM,
  title={Limites de repr{\'e}sentations cristallines},
  author={Berger, Laurent},
  journal={Compositio Mathematica},
  volume={140},
  number={6},
  pages={1473--1498},
  year={2004},
  publisher={London Mathematical Society}
}

@article{KisinEssDim,
 author = {Farb, Benson and Kisin, Mark and Wolfson, Jesse},
 title = {Essential dimension via prismatic cohomology},
 fjournal = {Duke Mathematical Journal},
 journal = {Duke Math. J.},
 issn = {0012-7094},
 volume = {173},
 number = {15},
 pages = {3059--3106},
 year = {2024},
 language = {English},
 doi = {10.1215/00127094-2023-0071}
}

@article{MorrowNotes,
  title={Notes on the {${A}_{\inf}$}-cohomology of {I}ntegral {$p$}-adic {H}odge theory},
  author={Morrow, Matthew},
  journal={arXiv preprint arXiv:1608.00922},
  year={2016}
}

@article{Banach,
 author = {Banach, St.},
 title = {Sur les op{\'e}rations dans les ensembles abstraits et leur application aux {\'e}quations int{\'e}grales.},
 fjournal = {Fundamenta Mathematicae},
 journal = {Fundam. Math.},
 issn = {0016-2736},
 volume = {3},
 pages = {133--181},
 year = {1922},
 language = {French},
 doi = {10.4064/fm-3-1-133-181},
 url = {https://eudml.org/doc/213289}
}
\bibliographystyle{amsalpha}.
\end{document}